\documentclass[12pt]{amsart} 

\usepackage[dvips]{graphics} 
\usepackage{color,amssymb,amsbsy,amsmath,amsfonts,amssymb,
amscd,epsfig,times,graphics,bm,enumerate}

\newtheorem{theo}{Theorem}
\newtheorem{lem}{Lemma}[section]
\newtheorem{rem}{Remark}[section]
\theoremstyle{definition}
\newtheorem{defi}{Definition}[section]
 
\numberwithin{equation}{section}

\newcommand{\C}{\mathbb C} 
\newcommand{\R}{\mathbb R}

\begin{document}
  
\begin{abstract}
We give a description of complex geodesics and we study the structure of stationary discs in some non-convex domains for which complex geodesics are not unique.
\end{abstract} 

\title[Invariant holomorphic discs in some non-convex domains]{Invariant holomorphic discs in some non-convex domains} 

\author{Florian Bertrand and Herv\'e Gaussier}

\thanks{Research of the two authors was partially supported by the CEDRE Grant 35398TK}
\maketitle 

\section{Introduction and Preliminaries}
The theory developed by L.Lempert in his seminal paper \cite{lem} offers a complete understanding of the complex geometry of bounded smooth strongly convex domains in $\mathbb C^n$: every such domain admits a singular foliation through any point by images of holomorphic discs centered at the point, these discs being infinitesimal extremal for the Kobayashi metric and stationary. Moreover, these discs are complex geodesics, are smooth up to the boundary, and are isolated. Some examples of strictly pseudoconvex non-convex domains with locally non-isolated extremal discs are presented in \cite{pa}. The aim of this article is to describe precisely (infinitesimal) extremal discs, complex geodesics (Theorem~\ref{theomain}), and to study the structure of stationary discs (Theorem~\ref{struct-theo}), for those  domains.

For $r>0$ we denote by $\Delta_r$ the disc centered at the origin of radius $r$ in $\C$ and let $\Delta=\Delta_1$ be the unit disc in $\C$.   
Let $\Omega \subset \C^n$ be a domain. The {\it Kobayashi pseudometric} $K_{\Omega}$ at $p\in \Omega$ and $v \in T_p\Omega$ is defined by
$$K_{\Omega}\left(p,v\right)=\inf 
\left\{r>0 \ | \  \exists \ f: \Delta \to \Omega \ \mbox{holomorphic}, f\left(0\right)=p, f'(0)=v/r\right\}$$
and the {\it Kobayashi pseudodistance} $d_{\Omega}$ is defined, for $p,q \in \Omega$, as its integrated pseudodistance
$$d_{\Omega}\left(p,q\right)=\inf\left\{\int_0^1 
K_{\Omega}\left(\gamma\left(t\right),\dot{\gamma}\left(t\right)\right)dt \ | \   
\gamma: [0,1]\rightarrow \Omega, \mbox{ }\gamma\left(0\right)=p, \gamma\left(1\right)=q\right\},$$
where the infimum is taken over all piecewise smooth $\mathcal C^1$ curves. When $d_{\Omega}$ is a distance, the domain $\Omega$ is called {\it Kobayashi hyperbolic}. Recall that in the case of the unit disc, then 
$K_\Delta$ and $d_\Delta$ are respectively the Poincar\'e metric and the Poincar\'e distance. We refer to \cite{kob} for the definitions and the main properties of the Kobayashi pseudometric, pseudodistance and of hyperbolic spaces.
Following \cite{ves, lem}, we may define complex geodesics and extremal discs as follows:
\begin{defi} Assume that $\Omega \subset \C^n$ is a Kobayashi hyperbolic domain and let $f: \Delta \rightarrow \Omega$ be a holomorphic map, also called holomorphic disc.  
\begin{enumerate}[(i)]
\item The disc $f$ is an {\it infinitesimal  extremal disc for the pair $(p,v) \in \Omega \times T_p\Omega$} if  $f(0)=p$, $f'(0)=\lambda v$ 
with $\lambda>0$ and if
 $g: \Delta \rightarrow \Omega$ is  holomorphic and such that $g(0)=p$, $g'(0)=\mu v$ with $\mu>0$, then $\mu\leq \lambda$.
\item The disc $f$ is an {\it infinitesimal complex geodesic} if $K_\Omega(f(\zeta),d_\zeta f(v_0)))=K_\Delta(\zeta,v_0)$ for all $\zeta \in \Delta$ and $v_0 \in \C$. 
 \item The disc $f$ is an {\it  extremal disc for the points $p,q \in \Omega$} if $f(0)=p$ and $f(\zeta)=q$ for some $\zeta \in \Delta$ and if  
 $g: \Delta \rightarrow \Omega$ is  holomorphic and such that $g(0)=p$ and $g(\zeta')=q$ then $|\zeta|\leq |\zeta'|$.

 \item The disc $f$ is a {\it complex geodesic} if $f$ is an isometry for the relative Kobayashi distances, namely 
 $d_\Omega(f(\zeta),f(\zeta'))=d_\Delta(\zeta,\zeta')$ for all $\zeta,\zeta' \in \Delta$.    
\end{enumerate}
\end{defi}
Note that in case the domain $\Omega$ is taut, then for any pair
$(p,v) \in \Omega \times T_p\Omega$   there exists an infinitesimal extremal disc for $(p,v)$. 
The question of the existence of complex geodesics is a difficult question. It was completely solved for bounded, smooth, strongly convex domains by L.Lempert in \cite{lem}. In case $\Omega$ is bounded and convex, according to to H.L.Royden and P.M.Wong \cite{ro-wo} (see also \cite{ab} Theorem 2.6.19), any (infinitesimal) extremal disc is a complex geodesic.
Such a domain being taut, this implies the existence of complex geodesics passing through any point in any direction in a bounded  convex domain. Moreover, in such domains $f$ is an infinitesimal complex geodesic if and only if $f$ is a complex geodesic 
(see \cite{ro-wo} and also \cite{ab} Corollary 2.6.20). 

We recall that a disc $f:\overline{\Delta} \to \C^n$, holomorphic in $\Delta$ and continuous up to $\partial \Delta$, is {\it attached} to a real hypersurface $M=\{\rho=0\} \subset \C^n$ if $f(\partial \Delta) \subset M$. 
Following \cite{lem}, such a disc is 
{\it stationary} for $M$ if there exists a continuous function $c:\partial \Delta\to\R\setminus \{0\}$ such that $\zeta \mapsto \zeta c(\zeta)\partial 
\rho(f(\zeta))$, defined on $\partial \Delta$, extends holomorphically on $\Delta$. Equivalently, following \cite{tum}, $f$ is {\it stationary} for $M$ if there exists a  holomorphic lift 
${\bm f}=(f,\tilde{f})$ of $f$ to the cotangent bundle $T^*\C^n$, continuous up to $\partial \Delta$ and such that ${\bm f}(\zeta)\in\mathcal{N}M(\zeta)$ for all $\zeta \in \partial \Delta$, where
\begin{equation}\label{eqco}
\mathcal{N}M(\zeta)=\{(z,\tilde{z}) \in \C^{2n}\  | \ z \in M, \tilde{z} \in \zeta N^*_z M\setminus \{0\} \}.
\end{equation}
Here $N^*_z M=span_{\R}\{\partial \rho(z)\}$ denotes the conormal fiber at $z$ of the hypersurface $M$. 

Finally, a set $X$ in a domain $\Omega \subset \C^n$ is a {\it holomorphic retract} if there exists a holomorphic map $r:\Omega \to \Omega$ such that 
$r(\Omega)\subset X$ and $r_{| X}=id_X$.

\vspace{2mm}
The following domain was introduced by N.Sibony. It is an example of a domain with non-isolated extremal discs, see \cite{pa} by M.-Y.Pang.    
Let $\rho$ be the real-valued function defined on $\C^2$ by
$$
\displaystyle \rho(z,w)=|z|^2+|w|^2-\Re e \left(\bar z ^4w^2\right)-1.
$$
We fix $\displaystyle 0<\varepsilon<\frac{1}{100}$ and we consider the domain $\Omega \subset \C^2$ defined by 
$$
\displaystyle \Omega=\{\rho<0\} \cap \left(\Delta_{1+\varepsilon} \times \Delta_{\frac{1}{4(1+\varepsilon)^3}}\right).
$$
\vspace{2mm}
One of the main purpose of the paper is to give a precise description of complex geodesics and of extremal discs contained in $\Omega$ and close to the disc $f^0: \Delta \to \Omega$ defined by 
$f^0(\zeta)=(\zeta,0)$. 

\vspace{2mm}
We observe the two following points:
\begin{enumerate}[(i)]
\item  The Levi form of $\rho$ at $(z,w) \in \mathbb C^2$ and $(Z,W) \in \mathbb C^2$ is given by 
\begin{eqnarray*}
\mathcal{L}\rho((z,w),(Z,W))&=&|Z|^2+|W|^2-8\Re e \left(\bar z ^3w \bar Z W\right)\\
\\
&=&  |Z-4\bar z ^3w W|^2+\left(1-16|z|^6|w|^2\right)|W|^2.
\end{eqnarray*}
For $\displaystyle |z|< 1+\varepsilon$ and $\displaystyle |w| < \frac{1}{4(1+\varepsilon)^3}$, the function $\rho$ is strictly plurisubharmonic. Therefore, the domain $\Omega$ is strongly pseudoconvex near 
 $\partial \Delta \times \{0\}$.  
\item The domain $\Omega$ is such that $\Omega \subset \Delta \times \C$. Indeed, if $(z,w) \in \Omega$ with $1<|z|$ then
$$|w|^2(1-|z|^4)\leq |w|^2- \Re e \left(\bar z ^4w^2\right)\leq 1-|z|^2$$
and thus 
  $$|w|^2(1+|z|^2) \geq 1$$
  which is not possible for $\displaystyle  (z,w) \in \Delta_{1+\varepsilon} \times  \Delta_{\frac{1}{4(1+\varepsilon)^3}}$.
 It follows that the set $\Delta \times \{0\}$ is a holomorphic retract of $\Omega$ and we denote by $\pi_1:\Omega \to \Delta$ the holomorphic projection defined by $\pi_1(z,w)=z$. 
\end{enumerate}

\section{Complex geodesics in $\Omega$}
Let $f:\Delta \to \Omega$ be a holomorphic disc of the form $f(\zeta)=(e^{i\theta}\zeta,f_2(\zeta))$, for some $\theta \in \R$. We first observe that by the decreasing property of the Kobayashi distance and metric, we have for $\zeta,\zeta' \in \Delta$ and $v \in \C$:
\begin{equation}\label{eqgeo}
d_{\Delta}(\zeta,\zeta')= d_{\Delta}(\pi_1(f(\zeta)),\pi_1(f(\zeta'))) \leq d_{\Omega}(f(\zeta),(f(\zeta')) \leq d_{\Delta}(\zeta,\zeta')
\end{equation}
and 
\begin{equation}\label{eqgeo2}
K_{\Delta}(\zeta,v)= K_{\Delta}(\pi_1(f(\zeta)),\pi_1(d_\zeta f v)) \leq K_{\Omega}(f(\zeta),d_\zeta f v) \leq K_{\Delta}(\zeta,v).
\end{equation}
This directly shows

\begin{lem}\label{lemgeo}
Any holomorphic disc $f:\Delta \to \Omega$ of the form $f(\zeta)=(e^{i\theta}\zeta,f_2(\zeta))$, for some $\theta \in \R$,  is a (infinitesimal) complex geodesic of $\Omega$.
\end{lem}
In particular we consider, as in \cite{pa}, the disc $f^t(\zeta)=(\zeta,t\zeta^2)$ for $t\geq 0$ small enough. Then the disc $f^t$ is a (infinitesimal) complex geodesic of $\Omega$.  We also note that, due to Lemma~\ref{lemgeo}, we have 
\begin{lem}\label{leminside}
Any holomorphic disc $f:\Delta \to \C^2$ of the form
$$
f(\zeta)=(e^{i\theta}\zeta,e^{2i\theta}\zeta(a_1+a_2\zeta+\overline{a_1}\zeta^2)),
$$
where $a_1 \in \C$, $a_2 \in \R, \theta \in [0,2\pi)$ are such that $\displaystyle 2|a_1|+|a_2|<\frac{1}{4(1+\varepsilon)^3}$ ,  is a (infinitesimal) complex geodesic of $\Omega$.
\end{lem}
\begin{proof}
We only need to prove that $f(\Delta) \subset \Omega$ when $a_1 \in \C$ and $a_2 \in \R$ are small enough. For $\zeta \in \partial \Delta$, we have  
$$
\begin{array}{lll}
\rho(f(\zeta))& = &|\zeta|^2+|\zeta(a_1+a_2\zeta+\overline{a_1}\zeta^2)|^2-\Re e \left(\bar \zeta ^4\zeta^2(a_1+a_2\zeta+\overline{a_1}\zeta^2)^2\right)-1\\
\\
                    & = & |a_1+a_2\zeta+\overline{a_1}\zeta^2|^2-\Re e \left(\bar \zeta ^2(a_1+a_2\zeta+\overline{a_1}\zeta^2)^2\right) =0.\\
\end{array}
$$
By the maximum principle, $f(\Delta) \subset \Omega$ provided that $a_1$ and $a_2$ are small enough.
\end{proof}
Our main result is the following
\begin{theo}\label{theomain}  Fix $z_0 \in \Delta\setminus\{0\}$ and write $\displaystyle x_0=e^{-i\theta_0}z_0 \in (0,1)$, $\theta_0 \in [0,2\pi)$.

\begin{enumerate}[(i)]
\item Let $z_1 \in \C$ be such that 
 $\displaystyle |z_1|< \frac{|z_0|^2}{4(1+\varepsilon)^3}$. Complex geodesics $f$ of $\Omega$ such that $f(0)=(0,0)$ and $f(r_0)=(z_0,z_1)$ for some $0<r_0<1$  
 are exactly of the form 
\begin{equation}\label{discform}
\displaystyle f(\zeta)=\left(e^{i\theta_0}\zeta,  e^{i2\theta_0}\zeta\left(x_0\left(-b+\overline{b}x_0^2\right)+
\left((1-x_0^4)b+\frac{z_1}{z_0^2}\right)\zeta+x_0\left(-\overline{b}+bx_0^2\right)\zeta^2\right)\right)
\end{equation}
for some $b \in \Delta_{\varepsilon_1}$ where 
$ \displaystyle \varepsilon_1=\frac{1}{5}\left(\frac{1}{4(1+\varepsilon)^3}-\left|\frac{z_1}{z_0^2}\right|\right)>0$.  In particular $r_0=x_0$.  
\item  The set of complex geodesics $f$ contained in $\Omega$ such that $f(0)=(0,0)$ and such that $f(x_0)=(z_0,z_1)$, with  
$\displaystyle |z_1|< \frac{|z_0|^2}{4(1+\varepsilon)^3}$, forms a  smooth real manifold of dimension three.
\item Let $c \in \R$ be such that $\displaystyle 0\leq c <\frac{1}{16(1+\varepsilon)^3}$. Then any infinitesimal extremal disc for the pair $\left((0,0),(1,c z_0)\right)$ is a (infinitesimal) complex geodesic. 
\item Let $z_1 \in \C$ be such that 
 $\displaystyle |z_1|< \frac{|z_0|^2}{4(1+\varepsilon)^3}$. 
Then any extremal disc for the points $(0,0)$ and $(z_0,z_1)$ is a (infinitesimal) complex geodesic. 
\end{enumerate}
\end{theo}

\begin{proof}
We start with Point $(i)$. First note that due to Lemma \ref{leminside}, any disc of the form $(\ref{discform})$ is a complex geodesics of $\Omega$. Moreover, it is immediate that for such discs we have $f(0)=(0,0)$ and $f(x_0)=(z_0,z_1)$.  Conversely, 
let $f=(f_1,f_2)$ be a  complex geodesic with $f(0)=(0,0)$ and $f(r_0)=(z_0,z_1)$ for some $0<r_0<1$ . 
We have:  
$$
d_{\Delta}(0,r_0)=  d_{\Omega}(f(0),f(r_0)) = d_{\Omega}((0,0),(z_0,z_1)) = d_{\Omega}(\varphi(0),\varphi(x_0))$$
where $\varphi: \Delta \to \Omega$ is any complex geodesic of the form   $(\ref{discform})$. Hence
$$ d_{\Delta}(0,r_0) =  d_{\Delta}(0,x_0)=d_{\Delta}(0,z_0)=d_{\Delta}(0,f_1(r_0)).$$
 It follows that $f_1:\Delta \to \Delta$ is a complex geodesic of the Poincar\'e disc given by $f_1(\zeta)=e^{i\theta}\zeta$ for some 
$\theta \in [0,2\pi)$. Note that since $f_1(r_0)=z_0$, we have $\theta=\theta_0$ and $r_0=x_0$.     
In order to determine $f_2$, we write for $\zeta \in \Delta$ 
$$|f_1|^2+|f_2|^2-\Re e \left(\bar f_1 ^4f_2^2\right)-1=|\zeta |^2+|f_2|^2-\Re e \left(\bar \zeta ^4(e^{-i2\theta_0}f_2)^2\right)-1 < 0.$$
As $|\zeta|$ tends to $1$, we obtain:
 $$|f_2(\zeta)|^2 \leq \Re e \left(\bar \zeta ^4(e^{-i2\theta_0}f_2(\zeta))^2\right),\ {\rm for \ every}\ \zeta \in \partial \Delta.$$
Consider the Fourier expansion of $\displaystyle f_2(\zeta)=\sum_{n \geq 1} a_n\zeta^n=\zeta h(\zeta)$ and set $\displaystyle \tilde{h}=e^{-i2\theta_0}h$. We have   
\begin{equation}\label{eqh}
|\tilde{h}|^2(\zeta) \leq \Re e \left(\overline{\zeta}\tilde{h}(\zeta)\right)^2
\end{equation}
 for $\zeta \in \partial \Delta$. Setting $g(\zeta)=\overline{\zeta}\tilde{h}(\zeta)$, Equation (\ref{eqh}) becomes $\displaystyle |g|^2 \leq \Re e g^2$
 on $\partial \Delta$. This implies that $g(\zeta)$ is a real-valued for $\zeta \in \partial \Delta$ and hence of the form 
 $$g(\zeta)=a_1\overline{\zeta}+a_2+\overline{a_1}\zeta,\ {\rm for \ every}\ \zeta \in \partial \Delta,$$
 with $a_2 \in \R$. Therefore $f_2$ is given for every $\zeta \in \Delta$ by:
$$f_2(\zeta)=e^{i2\theta_0}\zeta(a_1+a_2\zeta+\overline{a_1}\zeta^2).$$
Since  
$$e^{i2\theta_0}x_0(a_1+a_2x_0+\overline{a_1}x_0^2)=z_1$$
a straightforward computation gives  $\displaystyle a_1=x_0\left(-b+\overline{b}x_0^2\right)$  and $\displaystyle a_2=(1-x_0^4)b+\frac{z_1}{z_0^2}$, where $b \in \C$ is small enough.

\vspace{0,3cm}

We now prove Point $(ii)$.
Denote by $\mathcal{G}_0$ the set of complex geodesics of $\Omega$ 
centered at the origin and by $\mathcal{A}$ the set of holomorphic discs in $\C^2$ continuous up to 
$\partial \Delta$. Define the following set
$$U_{z_0}=\left\{(z_1,b) \in \C^2   \ | \  \ 5|b|+\left|\frac{z_1}{z_0^2}\right|<\frac{1}{4(1+\varepsilon)^3} \ \mbox{ and }  \ (1-x_0^4)b+\frac{z_1}{z_0^2} \in \R \right\}.$$
Consider the map $\mathcal{F}: \C^2 \to \mathcal{A}$  
defined by 
$$\mathcal{F}(z_1,b)(\zeta)=\left(e^{i\theta_0}\zeta,  e^{i2\theta_0}\zeta\left(x_0\left(-b+\overline{b}x_0^2\right)+
\left((1-x_0^4)b+\frac{z_1}{z_0^2}\right)\zeta+x_0\left(-\overline{b}+bx_0^2\right)\zeta^2\right)\right).$$
The map $\mathcal{F}$ is smooth, one-to-one and, by Point $(i)$, the set $\mathcal{F}(U_{z_0})$ is included in $\mathcal{G}_0$. Note that the  map $\mathcal{H}:\mathcal{A} \to \C^2$ defined by 
 $$\mathcal{H}(f)=\left(f_2(z_0),\frac{-1}{1-x_o^4}\left(\frac{f_2'(0)e^{-i\theta_0}}{z_0}+\overline{f_2'(0)}e^{i\theta_0}z_0\right)\right)$$
  is  smooth and satisfies $\mathcal{F} \circ \mathcal{H}(f)=f$ for $f \in \mathcal{F}(U_{z_0})$. Moreover for $(z_1,b) \in U_{z_0}$, the differential map $d_{(z_1,b)}\mathcal{F}$ of $\mathcal{F}$ at $(z_1,b)$ is given, for every $(Z_1,B) \in T_{(z_1,b)}U_{z_0}$, by: 
 $$d_{(z_1,b)}\mathcal{F}(Z_1,B)(\zeta)=\left(0,  e^{i2\theta_0}\zeta\left(x_0\left(-B+\overline{B}x_0^2\right)+
\left((1-x_0^4)B+\frac{Z_1}{z_0^2}\right)\zeta+x_0\left(-\overline{B}+Bx_0^2\right)\zeta^2\right)\right).$$
Note that $d_{(z_1,b)}\mathcal{F}$ is one-to-one.  Therefore the map $\mathcal{F}$ is a  smooth diffeomorphism onto its image. It follows that the set $\mathcal{F}(U_{z_0})$ is a  smooth real manifold of dimension three in $\mathcal{G}_0$.
\vspace{0,3cm}

We  prove Point $(iii)$. 
Let $f=(f_1,f_2):\Delta \to \Omega$ be an infinitesimal extremal disc for  the pair  $\left((0,0),(1,c z_0)\right)$ and let $\lambda>0$ be such that $f'(0)=\lambda(1,cz_0)$.   
Consider the disc $g:\Delta \to \Omega$
$$g(\zeta)=\left(\zeta, c \zeta(z_0-(1+|z_0|^2)\zeta+ \overline{z_0}\zeta^2)\right).$$
Note that $g(0)=(0,0)$ and $g'(0)=(1,cz_0)$. Since $f$ is an infinitesimal extremal disc, we have: 
$$1\leq \lambda =f_1'(0).$$
Moreover since $f_1:\Delta \to \Delta$ satisfies $f_1(0)=0$, by the Schwarz Lemma we have $|f_1'(0)|\leq 1$. This proves that $f_1(\zeta)=\zeta$ and therefore, by  
Lemma \ref{lemgeo}, $f$ is a (infinitesimal) complex geodesic. 

\vspace{0,3cm}

Finally, we prove Point $(iv)$. Let $f=(f_1,f_2):\Delta \to \Omega$ be an extremal disc for  the points  $(0,0)$ and $(z_0,z_1)$. Let $\zeta_0 \in \Delta$ be such that $f(\zeta_0)=(z_0,z_1)$.  
Consider an arbitrary disc $g:\Delta \to \Omega$ of the form (\ref{discform}). Since $g(0)=(0,0)$, 
$g(x_0)=(z_0,z_1)$ and since $f$ is an extremal disc we have:
$$|\zeta_0|\leq x_0=|z_0|=|f_1(\zeta_0)|.$$
The Schwarz Lemma implies that $f_1(\zeta)=e^{i\theta}\zeta$ for some $\theta \in \R$, and therefore $f$ is a complex geodesic  by  
Lemma \ref{lemgeo}.

\end{proof}

\begin{rem}
Note that according to Theorem \ref{theomain} Point $(i)$,  there are infinitely many (ranges of) complex geodesics centered at zero and passing through any point of the form 
$(z_0,z_1)$ where $z_0 \in \Delta\setminus\{0\}$ and  $\displaystyle |z_1|< \frac{|z_0|^2}{4(1+\varepsilon)^3}$.
\end{rem}

\section{Stationary discs for deformations of $\Omega$}

Let $\lambda \geq 0$. Following \cite{pa}, we consider the domain $\Omega_\lambda$ defined near $\overline{\Delta} \times \{0\}$  by 
$$
\rho^\lambda=|z|^2+|w|^2-\lambda \Re e \left(\bar z ^4w^2\right)-1<0.
$$
Note that for every $\lambda \geq 0$, the disc $f^0(\zeta)=(\zeta,0)$ is stationary, with lift ${\bm f^0}=(\zeta,0,1,0)$ to the cotangent bundle.  For $\lambda <<1$, $\Omega_{\lambda}$ is a small $\mathcal C^2$ deformation of the unit ball in a neighborhood of $\Delta \times \{0\}$, and $\Omega=\Omega_1$. We recall that in the case of the unit ball, lifts of stationary discs that are close to the lift  
${\bm f^0}(\zeta)=(\zeta,0,1,0)$ of $f^0$ form a smooth manifold of real dimension eight. It might be natural to expect the same result for $\Omega_{\lambda}$ at least for $\lambda < < 1$. The following theorem gives the result for the whole family, except at the domain $\Omega$.
\begin{theo}\label{struct-theo}
Assume that $\lambda \neq 1$. The set of lifts of stationary discs for $\partial \Omega_\lambda$ close to the lift ${\bm f^0}$ forms a  smooth real manifold of dimension eight. 
\end{theo}
\begin{proof}
For all $\zeta \in \partial \Delta$ the submanifold $\mathcal{N}\partial \Omega_\lambda(\zeta)$ (see Equation (\ref{eqco})
of real dimension four may be defined near ${\bm f^0}(\zeta)=(\zeta,0,1,0)$ by four real defining functions  
$\tilde{\rho^\lambda}=(\rho_1^\lambda,\cdots,\rho_4^\lambda)$ given explicitly by 
\begin{equation*}
\left\{
\begin{array}{lll} 

\rho_1^\lambda(\zeta)(z,w) & = & |z|^2+|w|^2-\lambda\Re e \left(\bar z ^4w^2\right)-1\\
\\

\rho_2^\lambda(\zeta)(z,w) & = & \displaystyle \frac{i\tilde{z}}{\zeta(\overline{z}-2\lambda z^3\overline{w}^2)}-\frac{i\zeta\overline{\tilde{z}}}
{z-2\lambda \overline{z}^3w^2} \\
\\
\rho_3^\lambda(\zeta)(z,w) & = & \displaystyle \tilde{w} - \frac{\tilde{z}(\overline{w}-\lambda\overline{z}^4w)}{\overline{z}-2\lambda z^3\overline{w}^2}+
\overline{\tilde{w}}-\frac{\overline{\tilde{z}}(w-\lambda z^4\overline{w})}{z-2\lambda\overline{z}^3w^2}
\\
\\
\rho_4^\lambda(\zeta)(z,w) & = & \displaystyle i\tilde{w}-\frac{i\tilde{z}(\overline{w}-\lambda\overline{z}^4w)}{\overline{z}-2\lambda z^3\overline{w}^2}-
i\overline{\tilde{w}}+\frac{i\overline{\tilde{z}}(w-\lambda z^4\overline{w})}{z-2\lambda \overline{z}^3w^2}.\\
\end{array}
\right.
\end{equation*}
It follows that a holomorphic disc $f$ is stationary for $\partial \Omega_\lambda$ and admits a lift ${\bm f}=(f,\tilde{f})$ close to ${\bm f^0}=(\zeta,0,1,0)$ if and only if  for all  $\zeta \in \partial \Delta$
 $$\tilde{\rho^\lambda}(\zeta)(\tilde{f}(\zeta))=0.$$
In order to solve this nonlinear Riemann-Hilbert problem, we need to evaluate the partial indices and Maslov index of the linearized 
problem along the disc  ${\bm f^0}=(\zeta,0,1,0)$, namely of the map $\bm{f} \mapsto 2 \Re e (\overline{G}\bm{f})$, where $G(\zeta)$ is the 
following  $4 \times 4$ invertible matrix
\begin{equation*}
\begin{array}{lll} 
G(\zeta)&=&\left((\rho^\lambda)_{\overline{z}}(\bm{f^0}(\zeta)), (\rho^\lambda)_{\overline{w}}(\bm{f^0}(\zeta)),(\rho^\lambda)_{\overline{\tilde{z}}}(\bm{f^0}(\zeta)),
(\rho^\lambda)_{\overline{\tilde{w}}}(\bm{f^0}(\zeta))\right)\\
\\
&=&\left(\begin{matrix}
\zeta   &  0 &  0 &  0 \\
\\ \displaystyle
-i\zeta  & 0  &     -i & 0 \\
\\
0 &    \displaystyle-\zeta+\lambda \zeta^3 &  0 & 1\\
\\
0  &    \displaystyle- i\zeta-i\lambda \zeta^3 &    0 & -i\\
\end{matrix}\right).
\end{array}
\end{equation*}
According to J.Globevnik \cite{gl2}, we need to prove that the partial indices of $-\overline{G^{-1}}G$ are greater than or equal to $-1$.   Since 
permutations of  rows and columns change neither the partial indices nor the Maslov index, we  work with the following matrix
$$\displaystyle  G_1(\zeta)=\left(\begin{matrix}
 -i   & -i\zeta &  0   & 0 \\
\\
0  &  \zeta &  0&  0 \\
\\
0 &    0 &  -\zeta+\lambda\zeta^3 & 1\\
\\
0  &   0 &     - i\zeta-i\lambda\zeta^3 & -i\\
\end{matrix}\right).
$$
A direct computation gives  
$$\displaystyle  G_1(\zeta)^{-1}=\left(\begin{matrix}
 i   & -1 &  0   & 0 \\
\\
0  & \displaystyle \overline{\zeta} &  0&  0 \\
\\
0 &    0 & \displaystyle  -\frac{\overline{\zeta}}{2}  & \displaystyle \frac{i\overline{\zeta}}{2}\\
\\
0  &   0 &   \displaystyle  \frac{1}{2}+\frac{\lambda \zeta^2}{2} &\displaystyle  \frac{i}{2}-\frac{i\lambda \zeta^2}{2}\\
\end{matrix}\right), \ \ \  \mbox{  } 
\displaystyle \overline{ G_1(\zeta)^{-1}}=\left(\begin{matrix}
- i   & -1 &  0   & 0 \\
\\
0  & \zeta &  0&  0 \\
\\
0 &    0 & \displaystyle  -\frac{\zeta}{2}  & \displaystyle -\frac{i\zeta}{2}\\
\\
0  &   0 &   \displaystyle  \frac{1}{2}+\frac{\lambda\overline{\zeta}^2}{2} &\displaystyle  -\frac{i}{2}+\frac{i\lambda\overline{\zeta}^2}{2}\\
\end{matrix}\right),
$$
and thus 
$$-\overline{G_1(\zeta)^{-1}}G_1(\zeta)=\left(\begin{matrix}
1   & 2\zeta &  0   & 0 \\
\\
0  &  -\zeta^2 &  0&  0 \\
\\
0 &    0 & \lambda \zeta^4 & \zeta \\
\\
0  &   0 &  \displaystyle  \zeta(1-\lambda^2) & \displaystyle -\lambda \overline{\zeta^2}\\
\end{matrix}\right).
$$
Let $\kappa_1 \geq \ldots \geq \kappa_{4}$ be the partial indices of $-\overline {G_1^{-1}} G_1$, and let $\Lambda$ be the diagonal matrix with entries $\zeta^{\kappa_1}, \ldots, \zeta^{\kappa_4}$. According to \cite{ve}, there exists a smooth map $\Theta: \overline{\Delta} \to GL_4(\C)$, where $GL_4(\C)$ denotes the general linear group,  holomorphic on $\Delta$ and such that on $\partial \Delta$ 
$$-\Theta \overline {G_1^{-1}} G_1 = \Lambda \overline \Theta.$$ 
We denoted by $l = (l_1,\ldots, l_4)$ the last row of the matrix map $\Theta$. In particular we obtain  the following system 
\begin{equation*}
\left\{
\begin{array}{lll} 
l_1(\zeta)&=  & \zeta^{\kappa_4} \overline{l_1(\zeta)} \\ 
\\
2\zeta l_1(\zeta)-\zeta^2l_2(\zeta)&=  & \zeta^{\kappa_4} \overline{l_2(\zeta)} \\ 
\\
\lambda \zeta^4 l_3(\zeta) + \zeta(1-\lambda^2)l_4(\zeta)&=  & \zeta^{\kappa_4} \overline{l_3(\zeta)} \\ 
\\
\zeta l_3(\zeta) + \lambda\overline{\zeta^2}l_4(\zeta)&=  & \zeta^{\kappa_4} \overline{l_4(\zeta)} \\ 
\end{array}
\right.
\end{equation*}
In case $l_1\neq 0$ we get $\kappa_4\geq 0$ by holomorphy of $l_1$. If $l_1=0$ the second line leads 
to $\kappa_4\geq 0$ unless $l_2=0$. 
Now if $l_1=l_2= l_3=0$ we also get $l_4=0$ by the third line which contradicts the fact that $\Theta(\zeta)$ is invertible. Finally,  if $l_1=l_2=0$, by 
holomorphy of  $l_3$ and $l_4$, the third lines leads to $\kappa_4\geq 0$. This proves that all partial indices are nonnegative.  

Moreover the Maslov index $\kappa=\kappa_1+\cdots+\kappa_4$ is the winding number of $\det (-\overline{G_1(\zeta)^{-1}}G_1(\zeta))$, that is 
$$
\kappa=\frac{1}{2i\pi}\int_{b\Delta}\frac{\left(\det (-\overline{G_1(\zeta)^{-1}}G_1(\zeta))\right)'}{\det (-\overline{G_1(\zeta)^{-1}}G_1(\zeta) (\zeta))}{\rm d}\zeta=4.
$$

Finally, this  proves that the set of lifts of stationary discs for $\partial \Omega_\lambda$ close to the lift ${\bm f^0}$ forms a  smooth real manifold of dimension $\kappa+\dim_{\C}\C^4=8$. 
\end{proof}

\begin{rem}
The previous proof illustrates a discontinuous behaviour of (lifts of) stationary discs at $\lambda=1$. More precisely, the previous method fails for $\lambda=1$ since, in that case, one of the partial indices is $-2$. Note that M.-Y.Pang \cite{pa} already showed a discontinuous behaviour of extremal discs at $\lambda=1$.  Moreover, a direct computation shows that geodesics of Lemma \ref{leminside} are stationary for $\partial \Omega$. It follows that the set of stationary discs for $\lambda=1$ centered at the origin forms a variety of real dimension at least four. Recall that in the case of the unit ball, and so for $\lambda <<1$,  such discs form a variety of real dimension three.   
\end{rem}

\vskip 0,5cm
{\small
\noindent Florian Bertrand\\
Department of Mathematics,\\
American University of Beirut, Beirut, Lebanon\\
{\sl E-mail address}: fb31@aub.edu.lb\\
\\
Herv\'e Gaussier\\
Univ. Grenoble Alpes, CNRS, IF, F-38000 Grenoble, France\\
{\sl E-mail address}: herve.gaussier@univ-grenoble-alpes.fr\\
}

\end{document}